\documentclass[authoryear,preprint,noinfoline]{imsart}

\RequirePackage[OT1]{fontenc}
\RequirePackage{amsthm,amsmath,amssymb}
\RequirePackage[sort,round]{natbib}
\RequirePackage[breaklinks,colorlinks]{hyperref}

\startlocaldefs
\usepackage{verbatim}

\DeclareMathOperator{\rk}{rank}
\newcommand{\RRR}{\mathbb{R}}

\theoremstyle{plain}
\newtheorem{theorem}{Theorem}

\newtheorem{lemma}[theorem]{Lemma}

\newtheorem{example}[theorem]{Example}

\endlocaldefs

\begin{document}

\begin{frontmatter}
  
  \title{Finiteness of small factor analysis models}
  \runtitle{Finiteness of factor analysis}

\begin{aug}
\author{\fnms{Mathias}
  \snm{Drton}\ead[label=e2]{drton@uchicago.edu}} 
\and
\author{\fnms{Han}
  \snm{Xiao}\thanksref{t1}\ead[label=e1]{xiao@galton.uchicago.edu}} 
 
\thankstext{t1}{This material is based upon work supported by the National
  Science Foundation under Grant No. DMS-0746265.  Mathias Drton was also
  supported by an Alfred P. Sloan Fellowship.} 
 \runauthor{M.~Drton, H.~Xiao}

\affiliation{University of Chicago}
\address{
  Department of Statistics, 5734 S.~University Ave, Chicago, IL  60637\\ 
\printead{e2,e1}}
\end{aug}

\begin{abstract}
  We consider small factor analysis models with one or two factors.  Fixing
  the number of factors, we prove a finiteness result about the covariance
  matrix parameter space when the size of the covariance matrix increases.
  According to this result, there exists a distinguished matrix size
  starting at which one can determine whether a given covariance matrix
  belongs to the parameter space by determining whether all principal
  submatrices of the distinguished size belong to the corresponding
  parameter space.  We show that the distinguished matrix size is equal to
  four in the one-factor model and six with two factors.
\end{abstract}





\end{frontmatter}

\renewcommand{\labelenumi}{\arabic{enumi}.}


\section{Introduction}
\label{sec:introduction}

Suppose we observe a sample of multivariate normal random vectors and wish
to test whether their covariance matrix is diagonal.  The likelihood ratio
test for this problem involves the determinant of the sample correlation
matrix. Therefore, this test cannot be used if the sample size $n$ is
smaller than the number $p$ of entries in the random vectors, because the
sample correlation matrix will always be singular.  A nice way around this
problem was proposed in \cite{schott:2005}, where the sum of squared
pairwise sample correlations is used as a test statistics.  The relevant
distribution theory involves a central limit theorem in the paradigm where
$n,p\to\infty$ such that $p/n\to c\in (0,\infty)$.

Why is it possible to prove such a limit theorem and which situations are
candidates for development of similar results and associated statistical
techniques?  We believe that a fundamental aspect of these questions is a
property of finiteness.  In the above problem this property amounts to the
ability to determine whether a covariance matrix is diagonal by verifying
whether each principal $2\times 2$ submatrix is diagonal.  The squared
sample correlations summed up in the test statistics of \cite{schott:2005}
do exactly that.  In this paper we show that such finiteness structure
arises more generally in factor analysis models.

The factor analysis model for $p$ observed variables and with $m$ factors
is the set of multivariate normal distributions $\mathcal{N}_p(\mu,\Sigma)$
with arbitrary mean vector $\mu$ and a covariance matrix $\Sigma$ in the
set
\[
F_{p,m} = \left\{ \Delta+\Gamma\Gamma^t \::\: \Delta
  \text{ positive definite and diagonal},\; \Gamma\in\RRR^{p\times m}
\right\}.
\]
For background on this classical statistical model see, for instance,
\cite{anderson:1956,harman:1976}.  Note also that for $m=0$ it is
reasonable to define $F_{p,0}$ to be the set of positive definite and
diagonal $p\times p$ matrices.  In this paper we establish the following
result that resolves part of a conjecture in \cite{drton:ptrf}; see also
open problem 7.8 in \cite{drton:2009}.

\begin{theorem}
  \label{thm:main}
  Suppose $m\in\{0,1,2\}$, and let $\Sigma=(\sigma_{ij})$ be a positive
  definite matrix of size $p\times p$ with $p\ge 2(m+1)$.  Then $\Sigma$ is
  in $F_{p,m}$ if and only if the principal submatrix $\Sigma_{A,
    A}=(\sigma_{ij})_{i,j\in A}$ is in $F_{2(m+1),m}$ for each index set
  $A\subseteq \{1,\dots,p\}$ of size $2(m+1)$.
\end{theorem}

As motivated above, this result is of statistical interest as it suggests
that in a high-dimensional setting with large number of variables $p$ a
test of the factor analysis models can be carried out by testing smaller
marginal hypotheses concerning only $2(m+1)$ variables.  While our result
provides the theoretical basis for such tests, an appropriate distribution
theory still has to be worked out.  This, however, is a topic beyond the
scope of this 
note.

The remainder of the paper is devoted to the proof of
Theorem~\ref{thm:main}.  In Section~\ref{sec:one-factor-models}, we outline
our approach to the finiteness problem and resolve the one-factor case.  In
Section~\ref{sec:two-factor-models}, we tackle the more complicated case of
two factors.  Concluding remarks are given in Section~\ref{sec:conclusion}.

\section{Approach to the problem and one-factor models}
\label{sec:one-factor-models}

We first introduce some notational conventions.  
Suppose $I$ and $J$ are two index sets. In this paper $I$ is typically of the form
$[p]=\{1,2,\ldots,p\}$ or $\{2,3,\ldots,p\}$ with $p\in\mathbb{N}$.
For any finite set $A$, let $|A|$ denote its cardinality.
Suppose $\Lambda=(\lambda_{ij})_{i\in I,j\in J}$ is an $|I|
\times |J|$ matrix.  For $A \subset I$ and $B\subset J$, we let
$\Lambda_{A,B}=(\lambda_{ij})_{i\in A, j\in B}$ denote the $|A|\times|B|$
submatrix.  When using the complement of a set $A$ as an index set we write
$\setminus A$ as in
\[
\Lambda_{\setminus A, \setminus B} := \Lambda_{I\setminus A, J\setminus B}.
\]
As a further shorthand, we let $\setminus i :=\setminus\{i\}$, 
$\Lambda_{A} := \Lambda_{A,J}$ and $\Lambda_{,B} := \Lambda_{I,B}$.

If $\Sigma\in F_{p,m}$, then the representation
$\Sigma=\Delta+\Lambda\Lambda^t$ is not unique because we may multiply an
orthogonal matrix to $\Lambda$ from the right.  However, the diagonal
matrix $\Delta$, and thus, the positive semi-definite and rank $m$ matrix
$\Lambda\Lambda^t$ may be unique.  We will repeatedly use the following
lemma to establish such uniqueness.

\begin{lemma}
  \label{lem:1}
  Let $p \geq 2m+1$, and consider two $p\times p$ matrices
  $\Psi=(\psi_{ij})$ and $\Phi=(\phi_{ij})$ with the same off-diagonal
  entries and same rank $m$. For any $i \in [p]$, if we can find two
  disjoint subsets $A,B \subset [p]\setminus\{i\}$ of cardinality
  $|A|=|B|=m$ such that $\det(\Psi_{A,B}) \neq 0$, then
  $\psi_{ii}=\phi_{ii}$.
\end{lemma}

\begin{proof}
  Since $\Psi$ and $\Phi$ are both of rank $m$, the following two $(m+1)
  \times (m+1)$ minors are equal to zero:
$$
\left| \begin{array}{cc}
\psi_{ii} & \Psi_{i,B} \\
\Psi_{A,i} & \Psi_{A,B}
\end{array} \right| 
= \left| \begin{array}{cc}
\phi_{ii} & \Phi_{i,B} \\
\Phi_{A,i} & \Phi_{A,B}
\end{array} \right| = 0. 
$$
The above two minors are entry-wise the same except for $\psi_{ii}$ and
$\phi_{ii}$. Since $\det(\Psi_{A,B}) = \det(\Phi_{A,B}) \neq 0$, it follows
that $\psi_{ii} = \phi_{ii}$.
\end{proof}

The next lemma will provide a way to give an induction-based proof of
Theorem~\ref{thm:main}.

\begin{lemma}
  \label{lem:2}
  Let $p \geq 2m+3$, and suppose that $\Sigma\in\RRR^{p\times p}$ is a
  positive definite matrix that has all $(p-1)\times (p-1)$ principal
  submatrices in $F_{p-1,m}$.  Write
  \begin{equation*}
    \Sigma_{\setminus p,\setminus p} = \Delta + \Gamma \Gamma^t
    \quad\text{and}\quad 
    \Sigma_{\setminus 1,\setminus 1} = D + GG^t
  \end{equation*}
  with $\Delta$ and $D$ positive definite and diagonal,
  and $\Gamma,G\in\RRR^{(p-1)\times m}$. To see the correspondence with the original matrix $\Sigma$ clearly, label the rows of $D$ and $G$ by $2,3,\ldots,p$. Then the
  following two conditions imply that $\Sigma$ belongs to $F_{p,m}$:
  \begin{itemize}
  \item[(i)] The two matrices $\Gamma$ and $G$ satisfy $\Gamma_{\setminus
      1} = G_{\setminus p}$;
  \item[(ii)] There are disjoint subsets $B,C\subseteq
    \left([p]\setminus\{1,p\}\right)$  
    of cardinality $|B|=|C|=m$ such that $\det(\Sigma_{B,C}) \neq 0$.
  \end{itemize} 
\end{lemma}

\begin{proof}
  From $\Delta=(\delta_{ij})$, $\Gamma=(\gamma_{ij})$, $D$ and $G$ form the matrices
  \[
  \bar{\Lambda} = \begin{pmatrix}
      \delta_{11} & 0 \\
      0 & D
    \end{pmatrix}
    \quad \text{and} \quad  \bar{\Gamma} = \begin{pmatrix}
      \Gamma_1 \\
      G
    \end{pmatrix}.
  \]
  Let $\bar{\Sigma} = \bar{\Lambda} + \bar{\Gamma} \bar{\Gamma}^t$, which
  is a matrix in $F_{p,m}$. We claim that $\Sigma=\bar{\Sigma} \in
  F_{p,m}$.
 
  By assumption (i), $\Gamma$ and $G$ agree in $p-2$ rows and it holds
  that $\bar{\sigma}_{ij} = \sigma_{ij}$ except possibly for the pair
  $(i,j) = (1,p)$.  In order to show that $\bar{\sigma}_{1p} =
  \sigma_{1p}$, we use the index sets $B$ and $C$ from conditon (ii) to
  construct the following three $(m+1) \times (m+1)$ minors, which are all
  equal to zero:
  \begin{equation}
    \label{eq:three-minors}
    \left| \begin{array}{cc}
        \bar{\Sigma}_{1,C} & \bar{\sigma}_{1p} \\
        \bar{\Sigma}_{B,C} & \bar{\Sigma}_{B,p}
      \end{array} \right| 
    = \left| \begin{array}{cc}
        \Sigma_{1,C} & \bar{\sigma}_{1p} \\
        \Sigma_{B,C} & \Sigma_{B,p}
      \end{array} \right| 
    = \left| \begin{array}{cc}
        \Sigma_{1,C} & \sigma_{1p} \\
        \Sigma_{B,C} & \Sigma_{B,p}
      \end{array} \right| = 0.
  \end{equation}
  The first minor is zero because $\{1\}\cup B$ and $C\cup\{p\}$ are two
  disjoint subsets of cardinality $m+1$ and because $\bar\Sigma$ is in
  $F_{p,m}$.  Recall that $\rk(\bar{\Gamma} \bar{\Gamma}^t)=m$.  The last
  minor is zero for the same reason.  Recall that we assume that $p\ge
  2m+3$, whereas the first principal submatrix shown in
  (\ref{eq:three-minors}) involves only $2m+2$ different variables.  
  The middle minor is zero
  because it is entry-wise equal to the first minor. Since by assumption
  $\det(\Sigma_{B,C}) \neq 0$, Lemma~\ref{lem:1} yields that
  $\bar{\sigma}_{1p} = \sigma_{1p}$.
\end{proof}

We are now ready to study the case of $m=1$ factor.

\begin{theorem}[$m=1$]
  \label{thm:m1}
  A positive definite matrix $\Sigma\in\RRR^{p\times p}$ with $p\ge 4$
  belongs to $F_{p,1} \setminus F_{p,0}$ if and only if every $4
  \times 4$ principal submatrix belongs to $F_{4,1}$ and at least one $2 \times 2$
  principal submatrix does not belong to $F_{2,0}$.
\end{theorem}

\begin{proof}
  The `only if' part is immediate. We prove the `if' part by induction. The
  induction base is $p =4$ in which case the statement is vacuous. Now
  suppose that $p\ge 5$ and that $\Sigma$ has every principal $4 \times 4$
  submatrix in $F_{4,1}$ and at least one non-zero off-diagonal entry.  If
  no other off-diagonal entry is non-zero, then
  $\Sigma=\Delta+\Gamma\Gamma^t$ for a vector $\Gamma\in\RRR^p$ with
  exactly two non-zero entries. So, we may assume that $\Sigma$ has at
  least two non-zeros entries.  Without loss of generality, suppose that
  $\sigma_{23} \neq 0$ and $\sigma_{34} \neq 0$.
  
  By the induction hypothesis, $\Sigma_{\setminus p,\setminus p}$ and
  $\Sigma_{\setminus 1,\setminus 1}$ belong to $F_{p-1,1}$.  Thus we are
  able to write
  $$
  \Sigma_{\setminus p,\setminus p} = \Lambda + \Gamma\Gamma^t
  \quad\text{and}\quad \Sigma_{\setminus 1,\setminus 1} = D + GG^t
  $$
  with $\Gamma,G\in\RRR^{p-1}$. Therefore we have 2 representations of
  $\Sigma_{\setminus\{1,p\},\setminus\{1,p\}}$, namely,
  $$
  \Sigma_{\setminus\{1,p\},\setminus\{1,p\}} = \Lambda_{\setminus
    1,\setminus 1} + \Gamma_{\setminus 1} \Gamma_{\setminus 1}^t =
  D_{\setminus p,\setminus p} + G_{\setminus p}G_{\setminus p}^t.
  $$
  Note that again we assign row indices to the matrices $D$ and $G$ based on the correspondence to the original matrix $\Sigma$, so we use $D_{\setminus p,\setminus p}$ and $G_{\setminus p,\setminus p}$ instead of $D_{\setminus (p-1),\setminus (p-1)}$ and $G_{\setminus (p-1),\setminus (p-1)}$ respectively. 
  Since $\sigma_{23}=\gamma_{2}\gamma_3 \neq 0$ and
  $\sigma_{34}=\gamma_3\gamma_4 \neq 0$, we deduce that
  $\sigma_{24}=\gamma_{2}\gamma_4 \neq 0$.  Using these three non-zero
  entries of $\Sigma$ and applying Lemma~\ref{lem:1}, we know that $\Gamma_{\setminus
    1}\Gamma_{\setminus 1}^t = G_{\setminus p}G_{\setminus p}^t$.
  Therefore, $\Gamma_{\setminus 1}$ and $G_{\setminus p}$ can only differ
  by a sign.  Changing the sign of $G$ if necessary, we obtain that
  $\Gamma_{\setminus 1} = G_{\setminus p}$ and thus Lemma~\ref{lem:2} implies that
  the induction step goes through.
\end{proof}


\section{Two-factor models}
\label{sec:two-factor-models}

If $\Gamma\in\RRR^{p\times m}$ has rank $m$ and $G$ is another matrix in
$\RRR^{p\times m}$ that satisfies $GG^t=\Gamma\Gamma^t$, then $\Gamma=GQ$
for some orthogonal matrix $Q$ \cite[Lemma 5.1]{anderson:1956}.  However,
this holds more generally.

\begin{lemma}
  \label{lem:ortho-trans}
  If $\Gamma,G\in\RRR^{p\times m}$ are matrices with dimensions $p\ge m$
  that satisfy $GG^t=\Gamma\Gamma^t$, then $\Gamma=GQ$ for some orthogonal
  matrix $Q$.
\end{lemma}
\begin{proof}
  Suppose $\Gamma$ has rank $k<m$.  Then $\rk(G)=k$ also.  There are now
  orthogonal transformations $Q_1$ and $Q_2$ such that $\Gamma
  Q_1=(\Gamma',0)$ and $G Q_2=(G',0)$, where $\Gamma'$ and $G'$ are full
  rank $p\times k$ matrices.  According to the above fact, $\Gamma'=G'Q$
  for some orthogonal $k\times k$ matrix $Q$.  Let
  $Q_3=\text{diag}(Q,I_{m-k})$.  Then $Q_2Q_3Q_1^t$ is an orthogonal matrix
  and $\Gamma=GQ_2Q_3Q_1^t$.
\end{proof}

By Lemma~\ref{lem:ortho-trans}, if the matrix
$\Sigma_{\setminus\{1,p\},\setminus\{1,p\}}$ has a unique representation as
the sum of a positive definite and diagonal matrix plus a positive
semi-definite matrix of rank $m$, then conditon (i) in Lemma 2 can be
satisfied by applying an orthogonal transformation.  We thus need to study
when this representation is unique and how uniqueness may fail for $m=2$
factors.  We begin our discussion by considering a $5\times 5$ matrix.
This prepares us for an induction step from $p=6$ to $p=7$ because in this
step $\Sigma_{\setminus\{1,p\},\setminus\{1,p\}}$ is of size $ 5\times 5$.

\begin{lemma}
  \label{lem:m2-unique}
  Suppose $\Sigma$ is a $5 \times 5$ positive definite matrix that can be
  written as
  \begin{equation}\label{reps}
    \Sigma=\Delta+\Gamma\Gamma^t,
  \end{equation}
  where $\Delta$ is positive definite and diagonal, and
  $\Gamma\in\RRR^{p\times 2}$ has rank $2$. Let $k$ be the largest integer
  $n$ for which there is an index set $A \subset [5]$ of cardinality $n$
  with $\rk(\Gamma_A)=1$.  Then under the assumptions that $\Sigma$
  has no representation like (\ref{reps}) with $\rk(\Gamma) \leq 1$, and
  every row of $\Sigma$ contains at least one non-zero off-diagonal entry,
  it holds that $k\le 3$ and we have the following:
\begin{enumerate}
\item[(i)] If $k=1$ or $k=2$, then the representation in (\ref{reps}) is
  unique, that is, $\Gamma\Gamma^t=GG^t$ for any other representation
  $\Sigma=D+GG^t$.
\item[(ii)] If $k=3$, then after a permutation of rows we may assume that
  $\rk(\Gamma_{[3]})=1$.  Then there exists an orthogonal matrix $Q$ such
  that 
  \[
  \Gamma Q = \begin{pmatrix}
    \gamma_{11}' & 0 \\
    \gamma_{21}' & 0 \\
    \gamma_{31}' & 0 \\
    \gamma_{41} & \gamma_{42}' \\
    \gamma_{51} & \gamma_{52}'
  \end{pmatrix},
  \]
  where primes indicate entries that are necessarily non-zero.  Moreover,
  for any other representation $\Sigma=D+GG^t$, it holds that $G=(g_{ij})$,
  when brought into the same form as $\Gamma$ by an orthogonal transformation,
  shares the first column with $\Gamma$ and satisfies
  $g_{42}'g_{52}'=\gamma_{42}'\gamma_{52}'$.
\end{enumerate}
\end{lemma}

\begin{proof}
  Since every row of $\Sigma$ is assumed to contain a non-zero entry, no
  row of $\Gamma$ is zero.  Assuming that $\rk(\Gamma)\le 1$ is not
  possible in a representation of $\Sigma$, we must have that $k \leq 4$.
  If $k=4$, by applying an orthogonal transform, we can write $\Gamma$ as
  \[
  \Gamma = \begin{pmatrix}
    \Gamma_{[4],1} & 0 \\
    \gamma_{51} & \gamma_{52}
  \end{pmatrix}.
  \]
  But then we have a contradiction to our assumptions because we can
  represent $\Sigma$ as
  \[
  \Sigma =   \begin{pmatrix}
    \Delta_{[4]} & 0\\
    0 &\delta_{55}+\gamma_{52}^2
  \end{pmatrix} +
  \begin{pmatrix}
    \Gamma_{[4],1} \\
    \gamma_{51}
  \end{pmatrix}
  \begin{pmatrix}
    \Gamma_{[4],1} \\
    \gamma_{51}
  \end{pmatrix}^t.
  \]
  Hence, we must have $k \leq 3$.
  
  Now suppose there is another representation $\Sigma = D + GG^{t}$.  Let
  $\Psi=\Gamma\Gamma^{t}$ and $\Phi=GG^{t}$.  There are two cases.
  \begin{itemize}
  \item[(a)] \underline{$k=1,2$:} If $k=2$, we may assume
    $\rk(\Gamma_{[2]})=1$.   
    By applying an orthogonal transform, we can write $\Gamma$ as 
    \[
    \Gamma = \begin{pmatrix}
      \gamma_{11}' & 0 \\
      \gamma_{21}' & 0 \\
      \gamma_{31}' & \gamma_{32}' \\
      \gamma_{41} & \gamma_{42}' \\
      \gamma_{51} & \gamma_{52}'
    \end{pmatrix},
    \]
    where we use again primes to highlight non-zero entries.  Since
    $k<3$, the submatrix $\Gamma_{\{3,4,5\}}$ has rank $2$,
    and thus we have enough non-zero $2 \times 2$ off-diagonal minors of
    $\Phi$ to apply Lemma~\ref{lem:1} and deduce uniqueness of the
    representation.  If $k=1$, then Lemma~\ref{lem:1} applies immediately.
  \item[(b)] \underline{$k=3$:} Let us assume $\rk(\Gamma_{[3]})=1$, in
    which case we can write $\Gamma$ as
    \[
    \Gamma = \begin{pmatrix}
      \gamma_{11}' & 0 \\
      \gamma_{21}' & 0 \\
      \gamma_{31}' & 0 \\
      \gamma_{41} & \gamma_{42}' \\
      \gamma_{51} & \gamma_{52}'
    \end{pmatrix}.
    \]
    Since the off-diagonal minor
    \[
    \left| \begin{array}{cc}
        \psi_{32} & \psi_{34} \\
        \psi_{52} & \psi_{54}
      \end{array} \right| \neq 0,
    \]
    Lemma~\ref{lem:1} yields that $\psi_{11} = \phi_{11}$. Similarly, we
    obtain that $\psi_{22} = \phi_{22}$ and $\psi_{33} = \phi_{33}$.  It
    follows that $\rk(G_{[3]})=1$, and thus we can write $G$ as
    \[
    G =
    \begin{pmatrix}
      \gamma_{11}' & 0 \\
      \gamma_{21}' & 0 \\
      \gamma_{31}' & 0 \\
      g_{41} & g_{42}' \\
      g_{51} & g_{52}'
    \end{pmatrix}.
    \]
    Since $\psi_{14}=\gamma_{11}'\gamma_{41}=\phi_{14}=\gamma_{11}'g_{41}$
    and $\psi_{15}=\gamma_{11}'\gamma_{51}=\phi_{15}=\gamma_{11}'g_{51}$,
    we know $\gamma_{41}=g_{41}$ and $\gamma_{51}=g_{51}$.  Therefore,
    \[
    G = 
    \begin{pmatrix}
      \gamma_{11}' & 0 \\
      \gamma_{21}' & 0 \\
      \gamma_{31}' & 0 \\
      \gamma_{41} & g_{42}' \\
      \gamma_{51} & g_{52}'
    \end{pmatrix}.
    \]
    For this case the representation is not unique, but since
    $\psi_{45}=\phi_{45}$ it must hold that
    $\gamma_{42}'\gamma_{52}'=g_{42}'g_{52}'$.  
\end{itemize}
\end{proof}

Equipped with Lemma~\ref{lem:m2-unique}, we are able to prove finiteness
for $m=2$ factors.

\begin{theorem}[$m=2$]
  A positive definite matrix $\Sigma\in\RRR^{p \times p}$ with $p \geq 6$
  belongs to $F_{p,2} \setminus F_{p,1}$ if and only if every principal $6
  \times 6$ submatrix belongs to $F_{6,2}$ and at least one $4 \times 4$
  principal submatrix does not belong to $F_{4,1}$.
\end{theorem}

\begin{proof}
  As in the proof of Theorem~\ref{thm:m1}, only the induction step requires
  work.  So suppose that every principal $6 \times 6$ submatrix of $\Sigma$
  is in $F_{6,2}$.  By the induction hypothesis, all the $(p-1) \times
  (p-1)$ principal submatrices belong to $F_{p-1,2}$, and without loss of
  generality, we may assume that $\Sigma_{\{2,\ldots,6\},\{2,\ldots,6\}}$
  contains a $4 \times 4$ submatrix not in $F_{4,1}$. 
  
  Suppose that some row of $\Sigma_{\{2,\ldots,6\},\{2,\ldots,6\}}$, say
  the first one, has all off-diagonal entries equal to zero. Consider the
  representation $\Sigma_{\setminus p,\setminus p} = \Lambda + \Gamma
  \Gamma^t$.  Then $\rk(\Gamma_{\{3,4,5,6\}\times[2]})=2$, for otherwise
  $\Sigma_{\{2,\ldots,6\},\{2,\ldots,6\}}$ would be in $F_{5,1}$. It
  follows that $\Gamma_2=(0,0)$ and all the off-diagonal entries in the
  second row of $\Sigma$ except for the last one are zero.  By considering
  a representation of the matrix $\Sigma_{\setminus 1,\setminus 1}$
  instead, we can deduce that in fact all of the off-diagonal entries in
  the second row of $\Sigma$ are zero.  Hence, the induction step goes
  through easily as we can insert a row of zeros into the matrix $\Gamma$
  in a representation $\Sigma_{\setminus 2,\setminus 2} = \Lambda + \Gamma
  \Gamma^t$.
  
  In the remaining cases, we can assume that the matrix
  $\Sigma_{\{2,\ldots,6\},\{2,\ldots,6\}}$ satisfies the two conditions in
  Lemma~\ref{lem:m2-unique}.  If $\Sigma_{\{2,\ldots,6\},\{2,\ldots,6\}}$
  belongs to case (i) of Lemma~\ref{lem:m2-unique}, then the center matrix
  $\Sigma_{\setminus\{1,p\},\setminus\{1,p\}}$ has a unique representation.
  To see this, note that by Lemma~\ref{lem:m2-unique},
  $\Sigma_{\{2,\ldots,6\},\{2,\ldots,6\}}$ has a unique representation, and
  that we can use a non-zero off-diagonal $2 \times 2$ minor from
  $\Sigma_{\{2,\ldots,6\},\{2,\ldots,6\}}$ to deduce the uniqueness of the
  representation of $\Sigma_{\setminus\{1,p\},\setminus\{1,p\}}$.  Therefore,
  Lemma~\ref{lem:2} implies that the induction step goes through in this
  case.  
  
  Now assume that $\Sigma_{\{2,\ldots,6\},\{2,\ldots,6\}}$ belongs to case
  (ii) of Lemma~\ref{lem:m2-unique}.  We first write $\Sigma_{\setminus
    p,\setminus p} = \Lambda + \Gamma \Gamma^t$ and $\Sigma_{\setminus
    1,\setminus 1} = D + GG^t$, where $\Gamma$ and $G$ are $(p-1) \times 2$
  matrices.  By Lemma~\ref{lem:m2-unique}, $\Gamma$ and $G$ have the
  following typical forms
  \[
  \Gamma = \begin{pmatrix}
    \gamma_{11} & \gamma_{12} \\
    \gamma_{21}' & 0 \\
    \gamma_{31}' & 0 \\
    \gamma_{41}' & 0 \\
    \gamma_{51} & \gamma_{52}' \\
    \gamma_{61} & \gamma_{62}'  \\
    \gamma_{71} & \gamma_{72}  \\
    \vdots & \vdots\\
    \gamma_{p-1,1} & \gamma_{p-1,2}  \\
    &
    \end{pmatrix}
    \quad\text{and}\quad G = \begin{pmatrix}
      & \\
      \gamma_{21}' & 0 \\
      \gamma_{31}' & 0 \\
      \gamma_{41}' & 0 \\
      \gamma_{51} & g_{52}' \\
      \gamma_{61} & g_{62}' \\
      g_{71} & g_{72} \\
      \vdots & \vdots\\
      g_{p-1,1} & g_{p-1,2} \\
      g_{p,1} & g_{p,2}
    \end{pmatrix},
  \]
  where we assigned row indices based on  the correspondence to the rows
  in $\Sigma$.
  
  If at least one of the entries $\gamma_{72},\dots,\gamma_{p-1,2}$ or
  $g_{72},\dots,g_{p-1,2}$ is non-zero, then Lemma~\ref{lem:1} implies
  uniqueness of the representation of
  $\Sigma_{\setminus\{1,p\},\setminus\{1,p\}}$, which allows to apply
  Lemma~\ref{lem:2}.  Otherwise, we only know that
  $\gamma_{52}'\gamma_{62}'=g_{52}'g_{62}'$. If
  $(\gamma_{52}',\gamma_{62}') = (g_{52}',g_{62}')$ or $\gamma_{12}=0$ or
  $g_{p,2}=0$, then the induction step goes through.  So we are left with
  the case, where $\gamma_{72}= \cdots = \gamma_{p-1,2} = g_{72} = \cdots = g_{p-1,2} = 0$,
  and $\gamma_{12}\neq 0$, $g_{p,2} \neq 0$.  Multiplying
  the second column of $G$ by $-1$ if necessary, we can assume that
  $\gamma_{52}'$ and $g_{52}'$ have the same sign.  To complete the proof
  we will show that in this case $\gamma_{52}'=g_{52}'$.
  
  Consider the submatrix $S=\Sigma_{\{1,3,4,5,6,p\},\{1,3,4,5,6,p\}}$.
  From the representation of $\Sigma_{\setminus p,\setminus p}$, we obtain
  that
  \begin{align}
    \label{eq:Srep1}
    S_{\setminus 6,\setminus 6} &=
    \Delta_{\{1,3,4,5,6\},\{1,3,4,5,6\}} +
    \begin{pmatrix}
      \gamma_{11} & \gamma_{12}' \\
      \gamma_{31}' & 0 \\
      \gamma_{41}' & 0 \\
      \gamma_{51} & \gamma_{52}' \\
      \gamma_{61} & \gamma_{62}'
    \end{pmatrix}
    \begin{pmatrix}
      \gamma_{11} & \gamma_{12}' \\
      \gamma_{31}' & 0 \\
      \gamma_{41}' & 0 \\
      \gamma_{51} & \gamma_{52}' \\
      \gamma_{61} & \gamma_{62}'
    \end{pmatrix}^t.
  \end{align}
  Similarly, the representation
  \begin{align}
    \label{eq:Srep2}
    S_{\setminus 1,\setminus 1} &= D_{\{3,4,5,6,p\},\{3,4,5,6,p\}}+
    \begin{pmatrix}
      \gamma_{31}' & 0 \\
      \gamma_{41}' & 0 \\
      \gamma_{51} & g_{52}' \\
      \gamma_{61} & g_{62}' \\
      g_{p1} & g_{p2}'
    \end{pmatrix}
    \begin{pmatrix}
      \gamma_{31}' & 0 \\
      \gamma_{41}' & 0 \\
      \gamma_{51} & g_{52}' \\
      \gamma_{61} & g_{62}' \\
      g_{p1} & g_{p2}'
    \end{pmatrix}^t.
  \end{align}
  is inherited from the representation of $\Sigma_{\setminus 1,\setminus
    1}$.  Since $S \in F_{6,2}$, it also has a representation $S = C +
  FF^t$ as in (\ref{reps}) with $C=(c_{ij})$ diagonal. We label the rows
  and columns of $C$ by $\{1,3,4,5,6,p\}$ based on  the correspondence to the rows
  in $\Sigma$. Using the structure
  of the two representations of $S_{\setminus 1,\setminus 1}$ and
  $S_{\setminus 6,\setminus 6}$ in (\ref{eq:Srep1}) and (\ref{eq:Srep2}),
  we can deduce via Lemma~\ref{lem:1} that $\delta_{55} = c_{55} = d_{55}$;
  again we assign indices for $\delta_{55}$ and $d_{55}$ 
  according to the correspondence to rows in $\Sigma$.  It follows that
  $\gamma_{51}^2+\gamma_{52}^2=\gamma_{51}^2+g_{52}^2$.  Having assumed
  that $\gamma_{52}$ and $g_{52}$ have the same sign, we conclude that
  $\gamma_{52}=g_{52}$.
\end{proof}


\section{Conclusion}
\label{sec:conclusion}

Our main result, Theorem~\ref{thm:main}, shows that for $m\le 2$ factors
the covariance matrices of distributions in the factor analysis model
possess a finiteness structure.  This also of interest for recent work on
the algebraic geometry of the two-factor model \citep{sullivant:gb}. Our
proof uses only linear algebra and shows that, in the covered cases, the
distinguished matrix size for finiteness is $2(m+1)$, that is, one can
decide whether a covariance matrix belongs to $F_{p,m}$ by only looking at
$2(m+1) \times 2(m+1)$ principal submatrices.  Unfortunately, our arguments
seem difficult to extend to the cases with $m\ge 3$ as one would need to
show that larger off-diagonal minors do not vanish in certain situations.

In unpublished work concerning a closure of the model, \cite{draisma} shows
that finiteness holds also for an arbitrary number of factors $m$.
However, his method of proof does not provide the distinguished matrix size
at which finiteness occurs.  It is natural to conjecture that this matrix
size is equal to $2(m+1)$ in general.  The following example clarifies that
the distinguished matrix size cannot be smaller.

\begin{example}
  Consider the matrix
  \[
  \Sigma = \begin{pmatrix}
    2I_{m+1} & I_{m+1} \\
    I_{m+1} & 2I_{m+1} 
  \end{pmatrix}
  \]
  where $I_{m+1}$ is an $(m+1) \times (m+1)$ identity matrix.  Then every
  $(2m+1) \times (2m+1)$ principal submatrix belongs to $F_{2m+1,m}$.  For
  example, the following matrix obtained by deleting one row and one column
  of $\Sigma$ can be written as
  \[
  \begin{pmatrix}
    2I_{m} & 0 & I_{m} \\
    0 & 2 & 0 \\
    I_{m} & 0 & 2I_{m}
  \end{pmatrix}
  = \begin{pmatrix}
    I_{m} & 0 & 0 \\
    0 & 2 & 0 \\
    0 & 0 &  I_{m} 
  \end{pmatrix}
  + \begin{pmatrix}
    I_{m} \\
    0  \\
    I_{m}
  \end{pmatrix}
  \begin{pmatrix}
    I_{m} \\ 0 \\ I_{m}
  \end{pmatrix}^t.
  \]
  Nevertheless, $\Sigma \notin F_{2m+2,m}$, because the off-diagonal block
  $\Sigma_{[m+1], [2m+2]\setminus[m+1]}$ has rank $(m+1)$.
\end{example}

\bibliographystyle{plainnat}
\bibliography{finiteness}

\end{document}